\documentclass[reqno, 11pt]{amsart}
\usepackage[left=2cm, right=2cm, top=3cm, bottom=3cm]{geometry}
\usepackage[T1]{fontenc}
\usepackage[utf8]{inputenc}	
\usepackage[english]{babel}
\usepackage[babel]{csquotes}
\usepackage{amssymb}
\usepackage{amsmath}
\usepackage{amsthm}
\usepackage{latexsym}
\usepackage{xcolor}
\usepackage{mathrsfs}
\usepackage{bbm}
\usepackage{verbatim}
\usepackage{graphicx}
\usepackage{array}
\usepackage{tikz}
\usepackage{subcaption}
\usepackage{tabularx}
\usepackage{pgfplots}
\usepackage{float}
\usepackage{mathtools}
\usepgfplotslibrary{fillbetween}

\newcolumntype{M}[1]{>{\centering\arraybackslash}m{#1}}

\newcommand\dela[1]{}

\DeclareMathOperator{\divv}{\mathrm{div}}

\numberwithin{equation}{section}

\newcommand{\eps}{\varepsilon}

\usepackage[colorlinks=true, linkcolor=blue, citecolor=blue, urlcolor=blue]{hyperref}

\usepackage{faktor}
\usepackage[shortlabels]{enumitem}
\usepackage[ruled,vlined]{algorithm2e}

\definecolor{myblue}{RGB}{31,119,180}
\definecolor{myorange}{RGB}{255,127,14}
\definecolor{mygreen}{RGB}{44,160,44}
\definecolor{myred}{RGB}{214,39,40}
\definecolor{mypurple}{RGB}{148,103,189}
\definecolor{mybrown}{RGB}{140,86,75}

\newtheorem{theorem}{Theorem}[section]

\newtheorem{definition}[theorem]{Definition}
\newtheorem{lemma}[theorem]{Lemma}
\newtheorem{proposition}[theorem]{Proposition}
\newtheorem{remark}[theorem]{Remark}

\newtheorem{notation}[theorem]{Notation}

\newcommand\E{\mathbb E}
\newcommand\Po{\mathbb P}
\def\cB{\mathscr{B}}
\def\cE{\mathscr{E}}
\def\B{\mathcal{B}}
\def\C{\mathrm{C}}

\makeatletter
\DeclareFontFamily{OMX}{MnSymbolE}{}
\DeclareSymbolFont{MnLargeSymbols}{OMX}{MnSymbolE}{m}{n}
\SetSymbolFont{MnLargeSymbols}{bold}{OMX}{MnSymbolE}{b}{n}
\DeclareFontShape{OMX}{MnSymbolE}{m}{n}{
	<-6>  MnSymbolE5
	<6-7>  MnSymbolE6
	<7-8>  MnSymbolE7
	<8-9>  MnSymbolE8
	<9-10> MnSymbolE9
	<10-12> MnSymbolE10
	<12->   MnSymbolE12
}{}
\DeclareFontShape{OMX}{MnSymbolE}{b}{n}{
	<-6>  MnSymbolE-Bold5
	<6-7>  MnSymbolE-Bold6
	<7-8>  MnSymbolE-Bold7
	<8-9>  MnSymbolE-Bold8
	<9-10> MnSymbolE-Bold9
	<10-12> MnSymbolE-Bold10
	<12->   MnSymbolE-Bold12
}{}

\let\llangle\@undefined
\let\rrangle\@undefined
\DeclareMathDelimiter{\llangle}{\mathopen}%
{MnLargeSymbols}{'164}{MnLargeSymbols}{'164}
\DeclareMathDelimiter{\rrangle}{\mathclose}%
{MnLargeSymbols}{'171}{MnLargeSymbols}{'171}
\makeatother

 \pgfplotsset{compat=1.18}
\begin{document}	
	\title[Existence of ergodic invariant measures for 3D  Navier-Stokes eqs with superlinear noise]{Existence of ergodic invariant measures for 3D  Navier-Stokes equations with superlinear noise}
	
	\author{Zdzislaw Brzezniak, Benedetta Ferrario, Margherita Zanella}
    \address{Department of Mathematics, University of York,
			Heslington, York, YO105DD, UK}
    \email{zdzislaw.brzezniak@york.ac.uk}
    \address{Dipartimento di Scienze Economiche e Aziendali,
		Università di Pavia, 27100 Pavia, Italy}
	\email{benedetta.ferrario@unipv.it}
    \address{Dipartimento di Matematica,
		Politecnico di Milano, Via E.~Bonardi 9, 20133 Milano, Italy}
	\email{margherita.zanella@polimi.it}
	\subjclass[2020]{}
	\keywords{}
	
	\begin{abstract}
    We consider  the  Navier-Stokes equations in a smooth bounded domain $D\subset \mathbb R^3$.
It has been proved in \cite{BZZ26} that these equations have a unique solution - globally defined in time -
when a superlinear noise  is introduced and  the initial velocity belongs to the space $V$ of divergence-free $H^1_0(D;\mathbb R^3)$-vector fields.
Here we show that the same result holds when there is also an additive noise and a  deterministic forcing term.
Moreover, when the deterministic forcing term is  time-independent 
we prove the existence of at least one invariant measure for this stochastic equation.
This is obtained by the Krylov-Bogoliubov method, based on tightness and Feller property in $V$. Finally, we show that the set of invariant measures contains one ergodic invariant measure.
	\end{abstract}
	\maketitle

\section{Introduction}

 We consider a velocity field $u$
evolving on a bounded three-dimensional domain $D$
with no-slip Dirichlet boundary conditions, which moves according to the Navier-Stokes Equations (NSEs)
\begin{equation}\label{NS}
\begin{aligned}
\partial_t u(t,x) &- \nu \Delta u(t,x)  + (u(t,x) \cdot \nabla) u(t,x) + \nabla p(t,x)  = F(t,x), \;\; t\geq 0, x \in D
\\
\divv   u(t,x) &= 0, \;\; t\geq 0, x \in D \\
u(0,x)&=u_0(x).
\end{aligned}
\end{equation}
Here $\nu>0$ is the viscosity, $p$ is the pressure,  $F$ is the external force and $u_0$ is the initial velocity.
There is a huge literature on the subject on these  equations, starting with the seminal work by Leray \cite{Leray_1934} about the global existence of weak solutions and finishing with the important paper \cite{Alb+Brue+Col_20022}
by  Albritton et all about non-uniqueness of weak solutions to some forced equations, i.e. when $F\not=0$. There is also big literature about stochastic NSEs, i.e. those when  the force $F$ is random starting with a papers \cite{Bens+Tem_1973} by Bensoussan and Temam, Flandoli \cite{Fla94} and possibly finishing with a recent paper \cite{Cheskidov+Zeng+Zhang_2026} by Cheskidov et all.

The main objective of the present paper is a certain stochastic version of  equations \eqref{NS} where the random force $F$
depends  in a nonlinear way of the energy $\frac12 \int_D \vert \nabla u(t,x) \vert^2 \, {\rm d}x$  of the solution $u$. To be precise, we assume that the random force $F$  is of the following form
\begin{equation}\label{eqn-F-random}
   F=f+\phi(\vert \nabla u(t)\vert_{L^2}) u(t)\, \frac{\partial}{\partial t} w_0(t)
+\frac{\partial}{\partial t} \tilde{w}(t),
  \end{equation}
where $f$ is a deterministic  external force,  $w_0$ is a real Wiener process  and $\tilde{w}$ is an infinite dimensional Wiener process.
This type of noise is a generalization of the noise studied in \cite{BZZ26},  by the first named author Zhang and Zhu,  where $f=0$ and  $\tilde w=0$.
Earlier papers on this type of stochastic NSEs are
\cite{TW} by Tang and Wang, \cite{CL} by Crisan and Lang
and especially
\cite{HLL} by Hong et all, which motivated the study \cite{BZZ26}. See also the paper \cite{BMX25} by Bagnara et all  for similar techniques but for different equations.

Our main  objective is to study ergodic properties of such stochastic NSEs. The well posedness of our system can be proven in a similar way to the well-posedness in  \cite{BZZ26}, see Theorem \ref{ex_uniq_sol} and its proof in section \ref{proof_well_pos}.
 Our main results are: (i)  Theorem \ref{teo-Feller} that the semigroup generated by our NSEs \eqref{SNS} is Markov and Feller, and (ii)
 Theorem \ref{thm_existence_invmeas}
that there exists an invariant measure $\mu$  for the Markov semigroup from (i) such that    for every $p\in (0,1)$, $\int_{V}\|x\|_{D(A)}^p\, \mu({\rm d}x)< \infty$ (and hence $\mu(D(A))=1$).

\dela{
{\color{gray}
Now we want to consider a more general noise.
The interest in introducing the second additive noise is that we do not expect the unique invariant measure to be the concentrated measure $\delta_0$ as in \cite{BFMZ25} for the pure multiplicative noise with coefficient $\varphi(u)u$, which is a quite trivial invariant measure.
}}

\bigskip

The subject of existence of invariant measures for stochastic Navier-Stokes Equations  (SNSEs) is quite well understood now. The first paper was written by Flandoli \cite{Fla94} in 1994; this paper dealt with SNSEs in  bounded 2-dimensional domains and with additive noise. 
This result was extended by Flandoli and Gatarek in \cite{FG95}, where the authors proved the existence of a stationary solution  to  the 3D Navier-Stokes equations in bounded domains with multiplicative noise, in both the standard and the transport forms. 
For  the 2D Navier-Stokes equations in bounded domains they proved the existence of an invariant measure because "in this case the transition semigroup is well defined". 
Chow and Khasminski in \cite{CK97} have proved similar results for Stochastic Evolution Equations including SNSEs, driven by an additive noise.

The question of existence of invariant measures in unbounded, also 2-dimensional,  domains for stochastic NSEs driven by an additive noise was solved by the first named author and Y.~Li in \cite{BL06}. 
The method was based on proving the existence of  compact invariant sets via asymptotical-compactness idea of Ball from \cite{Ball2004}, and on a result by Crauel \cite[Corollary 4.4]{CF94}, who proved that a compact invariant set, in the framework of random dynamical systems, supports an invariant measure for the corresponding Markov process. 
A more general result about the existence of invariant measures in unbounded 2-dimensional  domains for the stochastic NSEs driven by a multiplicative noise was presented by the first named author, Motyl and Ondrej\'at in \cite{BMO17}. 
In that paper they proved the existence of an invariant measure for the general 2D Navier-Stokes equations in unbounded domains satisfying the so-called Poincar\'e condition.  
Later on, the first and the second named authors generalized both  papers \cite{FG95} \cite{BMO17} and studied the existence of a stationary solution  to  the 3D Navier-Stokes equations in unbounded domains, see \cite{BF19}. 
The last two papers use a beautiful idea of  Maslowski and Seilder from \cite{Masl+S_1999} about Krylov-Bogoliubov method in a weak topology. 
Apart from the question of existence of invariant measures, a very important problem is of  the uniqueness of the invariant measure. 
This is a much more difficult problem and it was first studied by Flandoli and Maslowski \cite{FM95} and later by Ferrario \cite{Fer97, Fer99}. 
Finally, very important results were published by Hairer and Mattingly in 2006 \cite{HM06} and in 2011 \cite{HM11} about the uniqueness of the invariant measure for the 2D Navier-Stokes equations on the torus with degenerate noise.

There are also papers about invariant measures for various modifications of the stochastic 3-dimensional Navier-Stokes equations. Let us concentrate on one of them, usually called the tamed Navier-Stokes Equations or Brinkman--Forchheimer Equations, versions of the Navier-Stokes equations. We would like to mention three papers where the question of existence of invariant measures has been studied. One of them is by R\"ockner and Zhang \cite{RZ09}, where they proved the existence of an invariant measure for such equations on the torus, exploiting the fact that the equation was on the torus. The second  paper, an important generalization of the result by R\"ockner and Zhang,  is  by the first named author and Dhariwal \cite{BD20},  where the existence of invariant measures was proven for the whole space $\mathbb R^3$. A generalization of both these papers is discussed in a recent paper \cite{KCM2026}  by Kinra et al.

Recently, another modification of the stochastic 3-dimensional equations has become very popular, see the earlier mentioned papers \cite{TW}
\dela{by Tang and Wang}, \cite{HLL}
\dela{by Hong et all} and \cite{BZZ26}\dela{ by the first named author et all}. The primary aim of those papers was to prove the global existence of the modified 3D SNSEs. This modification is also studied in the present paper. However, the principle objective of our work is to show  the existence of an invariant measure for these equations.
From the point of view of the existence of an invariant measure,  we should not forget to  mention the recent beautiful paper \cite{CL}, in which the  authors considered the existence of global solutions to 3D SNSEs in their own framework, namely for initial data in  Sobolev spaces of much higher regularity, and  claim that they plan to study the existence of an invariant measure in a  subsequent paper.

Similarly, questions of existence of invariant measures were studied for modified stochastic nonlinear Schr\"odinger equations, for instance in a recent paper by the present three authors and Maurelli \cite{BFMZ25}, where they proved that a certain modification by suitable superlinear noise of the Nonlinear Schr\"odinger Equation yields the global existence, that under stronger conditions on the modification there exists also an invariant measure, and that under much stronger conditions on the coefficient the unique invariant measure is the Dirac delta measure at the origin. Let us point out, that surprisingly, contrary to the paper \cite{BFMZ25}, our conditions for the existence of an invariant measure are just the same as the assumptions in \cite{BZZ26} for the existence of global strong solutions.

As far as the contents of the paper are concerned, in Section \ref{sect:prelim} we introduce the basic notations and assumptions. 
In Section \ref{sect:wellp} we present the result of global existence of a unique solution.
Then the existence of invariant measures is given in  Section \ref{sect:invmeas} and that of ergodic invariant measures in Section \ref{sect:ergmeas}.
Finally, in the Appendix we provide some technical results.

\bigskip

\section{Preliminaries}
\label{sect:prelim}

\subsection{Functional spaces and operators}
\label{subsec:functional-spaces}

We fix an open bounded subset $D\subset\mathbb R^3$. For $1\le p\le\infty$ we
consider the Lebesgue space $\mathbb L^p=L^p(D;\mathbb R^3)$, endowed with the
norm
\[
\|u\|_{\mathbb L^p}=\Big(\int_D|u(x)|^p\,{\rm d}x\Big)^{1/p}\ (1\le p<\infty),
\qquad
\|u\|_{\mathbb L^\infty}=\operatorname*{ess\,sup}_{x\in D}|u(x)|.
\]
For $s\in\mathbb N$ we denote by $\mathbb H^s=H^s(D;\mathbb R^3)$ the Hilbert
space of elements of $\mathbb L^2$ whose distributional derivatives up to
order $s$ also belong to $\mathbb L^2$; the scalar products in
$\mathbb L^2=\mathbb H^0$ and $\mathbb H^1$ are
\[
\langle u,v\rangle_{\mathbb L^2}=\sum_{j=1}^3\int_Du_j(x)v_j(x)\,{\rm d}x,
\qquad
\langle u,v\rangle_{\mathbb H^1}=\langle u,v\rangle_{\mathbb L^2}
+\langle\nabla u,\nabla v\rangle_{\mathbb L^2},
\]
where $\langle\nabla u,\nabla v\rangle_{\mathbb L^2}
=\sum_{j,k=1}^3\int_D\partial_{x_k}u_j(x)\,\partial_{x_k}v_j(x)\,{\rm d}x$.

\smallskip
\subsubsection{Divergence-free spaces.}
Since we deal with incompressible
fluids, we work with the usual spaces, see, e.g., the monograph \cite{temam}, 
\begin{align}
H&\coloneq \{u\in\mathbb L^2:\operatorname{div}u=0\text{ in }D,\ u\cdot n=0\text{ on
}\partial D\},\\
V&\coloneq \{u\in\mathbb H^1:\operatorname{div}u=0\text{ in }D,\ u=0\text{ on
}\partial D\},
\end{align}
$n$ being the outer unit normal to $\partial D$. The space $H$ carries the
Hilbert structure inherited from $\mathbb L^2$, with norm $\|\cdot\|_H$ and
scalar product $(\cdot,\cdot)_H$. 
Identifying $H$ with its dual gives the variational
structure
\[
V\subset H\subset V^*,
\]
with dense and compact embeddings.

\smallskip
\subsubsection{The Stokes operator.}
Let $\Pi:\mathbb L^2\to H$ be the
Leray--Helmholtz orthogonal projection, and define the Stokes operator $A$ by
\[
Au=-\Pi(\Delta u)\in H\qquad\text{for }u\in D(A)=\mathbb H^2\cap V.
\]
The domain $D(A)$ is a Hilbert space when endowed with the graph norm
\begin{equation}
\|u\|_{D(A)}^{2}:=\|u\|_{H}^{2}+\|Au\|_{H}^{2}, \quad u\in D(A);
\end{equation}
this norm is equivalent to the $\mathbb{H}^{2}$ norm. 

Since $A$ is a nonnegative self-adjoint operator on $H$, the spectral theorem
gives an orthonormal basis $\{e_k\}_{k\ge1}$ of $H$ and eigenvalues
$0<\lambda_1\le\lambda_2\le\cdots\to\infty$ with $Ae_k=\lambda_ke_k$; the
fractional powers of $A$ are then defined, for every $\alpha\in\mathbb R$, by
\[
D(A^\alpha):=\Big\{u=\sum_{k\ge1}u_ke_k\in H:\sum_{k\ge1}\lambda_k^{2\alpha}
|u_k|^2<\infty\Big\}.
\]
It follows that $D(A^{1/2})=V$. Thanks to the Poincar\'e-type inequalities
\begin{equation*}
\|u\|_H^2\le\frac1{\lambda_1}\|u\|_V^2,
\end{equation*}
\begin{equation}\label{Poincare-in-DA}
\|u\|_V^2 \le \frac1{\lambda_1}\|Au\|_H^2,
\end{equation}
that hold for every $u\in V$, respectively  $u\in D(A)$, we identify the
$V$-norm with $\|A^{1/2}u\|_H$ and the $D(A)$-norm with $\|Au\|_H$.

In this paper we consider $V$ as state space and therefore we deal with the restriction of the semigroup $S(t):=e^{At}$ to the space $V$. It is known that it is an analytic   semigroup in $V$.

\smallskip
\subsubsection{Path spaces.}
For $T>0$ we set
$X_T:=C([0,T];V)\cap L^2(0,T;D(A))$, endowed with
\[
\|u\|_{X_T}^2:=\sup_{t\in[0,T]}\|u(t)\|_V^2+\int_0^T\|Au(t)\|_H^2\,{\rm d}t,
\]
and denote by $\mathbb M^2(X_T)$ the separable Banach space of (equivalence
classes of) $D(A)$-valued progressively measurable processes $u$ whose
trajectories a.s.\ belong to $X_T$, with
$\|u\|_{\mathbb M^2(X_T)}^2:=\E\big(\|u\|_{X_T}^2\big)<\infty$.

\smallskip
\subsubsection{The nonlinear term.} We define the bilinear operator
$B(u,v):=\Pi\big((u\cdot\nabla)v\big)$ and set $B(u):=B(u,u)$. It satisfies
\begin{equation}\label{B-properties}
\langle B(u,v),w\rangle=-\langle B(u,w),v\rangle,\qquad
\langle B(u,v),v\rangle=0,\qquad u,v,w\in V,
\end{equation}

We will repeatedly use the following bounds,
proved respectively in \cite[Lemma~2.7]{BZZ26} and
\\
\cite[Proposition~7.5]{BZZ26} that  there exists $C>0$ such that, for every
$u,v\in D(A)$,
\begin{align}\label{prima-stime-B}
\|B(u)\|_H&\le C\|u\|_{D(A)}^{1/2}\|u\|_V^{3/2},
\\
\label{stima-diff-B}
\|B(u)-B(v)\|_H &\le C\Big(\|u\|_{D(A)}^{1/2}\|u\|_V^{1/2}\|u-v\|_V
+\|u-v\|_{D(A)}^{1/2}\|u-v\|_V^{1/2}\|v\|_V\Big).
\end{align}


\subsection{Stochastic framework and abstract formulation}
\label{subsec:stochastic}

Let $(\Omega,\mathcal F,(\mathcal F_t)_{t\ge0},\Po)$ be a filtered
probability space satisfying the usual conditions (saturated and
right-continuous filtration), carrying a sequence $\{w_j\}_{j=0}^\infty$ of
i.i.d.\ real-valued Wiener processes. Applying $\Pi$ to the Navier--Stokes
equation \eqref{NS} eliminates the pressure and yields the abstract
deterministic form
\[
\frac{{\rm d}u}{{\rm d}t}(t)+\nu Au(t)+B(u(t))=\Pi F,
\]
where, throughout this paper, $F$ is taken to be time-independent and we assume $\Pi F=f$.
Adding the stochastic forcing terms, a multiplicative noise
$g(u)=\varphi(u)u$ driven by $w_0$, and an additive noise driven by the
remaining $\{w_j\}_{j\ge1}$, we arrive at the stochastic Navier--Stokes
equation studied in this paper:
\begin{equation}\label{SNS}
{\rm d}u(t)+[Au(t)+B(u(t))]\,{\rm d}t= f\,{\rm d}t+g(u(t))\,{\rm d}w_0(t)
+G\,{\rm d}w(t),
\end{equation}
where
\[
Gw(t):=\sum_{j=1}^\infty\gamma_je_j^V\,w_j(t),\qquad e_j^V:=e_j/\sqrt{\lambda_j},
\]
$\{e_j^V\}_{j\ge1}$ being the orthonormal basis of $V$ induced by
$\{e_j\}_{j\ge1}$, and $\{\gamma_j\}_{j\ge1}\subset\mathbb R$.

\subsection{Assumptions}
\label{subsec:assumptions}

We consider a multiplicative noise coefficient $g(u)=\varphi(u)u$, with $\varphi:V\to\mathbb R$.
We denote by
\textbf{[A1]} the conjunction of the following three conditions:
\begin{itemize}
\item[A1a)] \emph{(Local Lipschitz continuity.)} For every $R>0$ there is
$C_R>0$ such that
\[
\|g(u)-g(v)\|_V\le C_R\|u-v\|_V,
\]
 for all $u,v\in V$ with
$\|u\|_V,\|v\|_V\le R$.
\item[A1b)] \emph{(Superlinear growth.)} There exist $\kappa>2$ and $b>0$
such that $|\varphi(u)|\ge b\|u\|_V^\kappa$ for every $u\in V$.
\item[A1c)] \emph{(Local boundedness on path space.)} For every $T,R>0$,
there is $C:=C(R,T)>0$ such that $\sup_{t\in[0,T]}|\varphi(u(t))|\le C$
whenever $u\in X_T$ satisfies $\|u\|_{X_T}\le R$.
\end{itemize}

\begin{remark}
Throughout the paper we work under the assumption $\kappa>2$ in [A1b)]. The
case $\kappa=2$ with $b$ sufficiently large, as considered in \cite{BZZ26},
could also be analyzed: how large $b$ depends on the structural
parameters of the problem, and this in turn later forces a restriction on the
admissible moments $p$ with respect to the invariant measure (see
Theorem \ref{thm_existence_invmeas}). We refer to Remark \ref{rem_k_2} for further details.
\end{remark}

For the additive noise, we denote by \textbf{[A2]} the summability condition
\begin{equation}\tag{A2}
\sum_{j=1}^\infty\gamma_j^2<\infty.
\end{equation}
In other words, we assume that $Gw$ is an $V$-valued Wiener process.

For the deterministic forcing, we denote by \textbf{[A3]} the condition
\begin{equation}\tag{A3}
f\in H.
\end{equation}

As an immediate consequence of [A1a)] and [A1c)], we obtain the following result.
\begin{lemma}\label{lemA}
Assume [A1a)] and [A1c)]. Then for every $\delta>0$ there exists
$K_\delta>0$ such that
\begin{equation}\label{A1c-ineq}
|\varphi(u)|^2\|u\|_V^2\ \le\ \delta\,|\varphi(u)|^2\|u\|_V^4
+K_\delta \|u\|_V^2\quad \text{ for every $u\in V$}
\end{equation}
\end{lemma}
\begin{proof}
Assumption  [A1a)] with $v=0$ becomes
$|\varphi(u)| \|u\|_V = \|g(u) \|_V \le C_R\|u\|_V$; therefore
we get that for
every $R>0$ there exists $C_R>0$ such that
\begin{equation}\label{phi-bounded}
\sup_{\|u\|_V\le R} |\varphi(u)| \le C_R,
\end{equation}
that is $\varphi$ is bounded on bounded subsets of $V$.

Now, consider inequality \eqref{A1c-ineq}, rewritten as
\[
|\varphi(u)|^2\|u\|_V^2(1- \delta\|u\|_V^2)
\le
K_\delta \|u\|_V^2.
\]
If $1- \delta\|u\|_V^2\le 0$, then the l.h.s. is not positive and so the inequality is trivially fulfilled. On the other side, if $1- \delta\|u\|_V^2> 0$, that is $\|u\|_V<\frac 1 {\sqrt \delta}$, we use the local  boundedness \eqref{phi-bounded} of $\varphi$ to conclude that the l.h.s. is bounded:
\[
|\varphi(u)|^2\|u\|_V^2(1- \delta\|u\|_V^2)
\le
C_{1/\sqrt \delta}^2\|u\|_V^2(1- \delta\|u\|_V^2)
\le C_{1/\sqrt \delta}^2  \ \|u\|_V^2 .
\]
\end{proof}

\begin{notation}
Throughout the paper, we reserve the symbol $C$ (eventually indexed) for
constants depending on some or all of the structural parameters of the
problem; when a dependence is relevant, it is pointed out explicitly. The
value of $C$ may change within the same argument without relabelling.
\end{notation}


\section{Well-posedness result}
\label{sect:wellp}

The existence of a unique  solution, defined for every time, has been obtained in \cite{BZZ26}  for a spatial domain which can be  bounded or  unbounded, differently from \cite{HLL} where the spatial domain is assumed to be bounded. Both papers have only the superlinear noise term.
However the deterministic forcing term and the additive noise do not change the main result.

\begin{definition}
\label{global_sol}
Let $u_0$ be a $V$-valued $\mathcal F_0$-measurable random variable with
$\E\|u_0\|_V^2<\infty$. A \emph{(global) solution} to \eqref{SNS} with initial
datum $u_0$ is a process $u:[0,\infty)\times\Omega\to V $ such that
\begin{enumerate}[(i)]
\item $u$ is $\mathcal F$-adapted and has $\Po$-a.s. paths in $C([0,\infty);V)\cap L^2_{loc}(0,\infty,D(A))$, 
\item for every $t\ge0$, $\Po$-a.s.,
\begin{equation}\label{mild-global}
\begin{aligned}
u(t)&=S(t)u_0-\int_0^tS(t-r)B(u(r))\,{\rm d}r
+\int_0^tS(t-r)\, f\,{\rm d}r
\\
&+\int_0^tS(t-r)g(u(r))\,{\rm d}w_0(r)+\int_0^tS(t-r)G\,{\rm d}w(r).
\end{aligned}
\end{equation}
\end{enumerate}
A solution is \emph{unique} if, whenever $u,v$ are two solutions
with the same initial datum $u_0$, $u(t)=v(t)$ $\Po$-a.s.\ for every $t\ge0$.
\end{definition}

\begin{remark}
The solution is defined in mild form, but it also satisfies the strong
formulation in $H$,i.e.
\[
u(t)=u_0-\int_0^t\big[Au(r)+B(u(r))-f\big]\,{\rm d}r+\int_0^tg(u(r))\,{\rm d}w_0(r)+\int_0^tG\,{\rm d}w(r).
\]
For $G=0$ and $f=0$ this equivalence is proved in \cite{BZZ26} Proposition 3.6; this equivalence can be easily extended by including the deterministic forcing term and  the additive noise.
\end{remark}

This is main result on the well posedness of the 3D stochastic Navier-Stoeks equation.
\begin{theorem}
\label{ex_uniq_sol}
Assume [A1], [A2] and [A3]. Let $u_0$ be an $\mathcal F_0$-measurable $V$-valued random
variable with $\E\|u_0\|_V^2<\infty$. Then there exists a unique global solution,
in the sense of Definition \ref{global_sol} to the
stochastic Navier-Stokes equation \eqref{SNS}.
\end{theorem}

When $G=0$ and $f=0$, the proof has been done in \cite{BZZ26}. To generalize it in our  case with  an additive noise and a deterministic forcing term one can follow the strategy of \cite[\S 4--8]{BZZ26}: the details are given in Appendix \ref{proof_well_pos}.
An important observation is that the presence of the deterministic and  random additive forces do not lead to a change in the assumption for the diffusion nonlinearity $\phi$.

\section{Existence of an invariant measure and properties of its support}
\label{sect:invmeas}

In this section we prove the existence of an invariant measure as well as some properties concerning its support.
As we shall see, it will be useful to introduce some \textit{ad-hoc} notation that we illustrate hereafter.
The symbol $\mathscr{B}(V)$ denotes the $\sigma$-algebra of all Borel subsets of $V$;  moreover, for a set $A \in \mathscr{B}(V)$, we denote by $A^c$ its complement. By $\mathscr{P}(V)$ we denote the set of all probability measures on $(V,\mathscr{B}(V))$. Moreover, $\mathcal{B}_b(V)$ and $\mathcal{C}_b(V)$ will stand for the space of all measurable bounded functions from $V$ to $\mathbb{R}$ and the space of continuous bounded functions from $V$ to $\mathbb{R}$, respectively.

For every $x\in V$, the unique solution
to equation \eqref{SNS} with initial datum $x$, as given in Theorem~\ref{ex_uniq_sol}, will be denoted by $u^x$,
and for every $t\ge 0$ we set  $u(t;x):=u^x(t)$ for its value at time $t$.
We can introduce
the family of operators $P:=(P_t)_{t \ge 0}$ associated to equation \eqref{SNS} as
\begin{equation}
\label{def-Pt}
(P_t\psi)(x):= \mathbb{E}[ \psi(u^x(t))], \quad x\in V,\quad \psi\in \B_b(V).
\end{equation}
It is clear that $P_t\psi$ is bounded
for every $\psi \in \B_b(V)$. Since we know from \cite[Corollary 23]{On2005}
that the transition function is jointly measurable, that is for every $A\in\cB(V)$
the map $V \times [0,\infty)\ni  (x,t)\mapsto \mathbb{P}(u^x(t)\in A)\in \mathbb{R}$ is measurable, it follows that $P_t\psi$ is also measurable for every $\psi \in \B(V)$.
Hence $P_t$ maps $\B_b(V)$ into itself for every $t \ge0$.

Furthermore, since the unique solution of \eqref{SNS} is a $V$-valued continuous process,
then it is also a Markov process, see \cite[Theorem 27]{On2005}.
Therefore we deduce that the family of operators $\{P_t\}_{t\ge0}$ is a Markov semigroup,
namely $P_{t+s}=P_tP_s$ for every $s,t\ge0$.

We are ready to give the precise definition of invariant measure.
\begin{definition}
  An invariant measure for the transition semigroup $P$
  is a probability measure $\mu\in\mathscr{P}(V)$ such that
  \[
  \int_V \psi(x)\,\mu({\rm d} x) = \int_V P_t\psi(x)\,\mu({\rm d} x) \quad\forall\,t\geq0,\quad\forall\,\psi\in\C_b(V).
  \]
\end{definition}

\begin{remark}
\label{rem_inv}
By a standard argument one proves that invariance for all $\psi\in\C_b(V)$ is equivalent to invariance for all
$\psi\in\B_b(V)$.
\end{remark}

This is our main result.
\begin{theorem}
\label{thm_existence_invmeas}
Assume [A1], [A2] and [A3].
\begin{itemize}
\item[(i)] There exists at least one invariant measure $\mu$  for the stochastic Navier-Stokes equation \eqref{SNS}.
\item [(ii)]
Moreover, for every $p\in (0,1)$, there exists a constant $C>0$,
depending only on the structural parameters, such that for
every invariant measure $\mu$ it holds
\begin{equation}
\label{int_mu_V}
\int_{V}\|x\|_{D(A)}^p\, \mu({\rm d}x)\le C
\end{equation}
and the support of $\mu$ is contained in $D(A)$, i.e. $\mu(D(A))=1$.
\end{itemize}
\end{theorem}
In what follows we will prove at first statement (i): the proof of the existence of an invariant measure relies on the Krylov–Bogoliubov Theorem, which is based on the Feller property of the  semigroup $P$  and the tightness property for time-averaged laws. We will then prove statement (ii) exploiting suitable apriori bounds on the solution process.

\subsection{Feller property}
In this section we prove the Feller property in $V$, that is $P_t: \C_b(V) \to  \C_b(V)$ for every $t\ge 0$.
This is a kind of continuity result with respect to the initial velocity.
\begin{theorem}\label{teo-Feller}
    Assume  [A1], [A2] and [A3]. Then the semigroup $P$ given in \eqref{def-Pt} is Feller.
\end{theorem}

The proof of the above result is based on two preliminary results stated as Lemmata below.
Given any two initial velocities $x, y$ in $V$ and a positive constant $M$
we introduce
the stopping times
 \begin{equation}
 \label{tau_M}
\tau^{x} _M=\inf\{t\in [0, + \infty):
\sup_{0\le s\le t} \|u^{x}(s)\|^2_V
+\int_0^t \|Au^{x}(r)\|_H^2{\rm d}r \ge M \},
    \end{equation}
 \begin{equation}\label{def-tau-n-M}
\tau^{x,y} _M=\min (\tau^{x} _M, \tau^{y} _M).
    \end{equation}

The first result provides a local continuity result in mean square.
\begin{lemma}\label{l1}
Consider two solutions $u^x$ and $u^y$ of the Navier-Stokes equation \eqref{SNS} with initial
velocities  $x,y\in V$, respectively, and the stopping time $\tau_M^{x,y}$
defined in \eqref{def-tau-n-M}.
for every $M>0$
there exists  a positive constant $C_M$  such that, for ant $T>0$ we have
\begin{equation}\label{differenza-in-media}
    \E\Big[ \sup_{t\in [0,T\wedge \tau_M^{x,y}]}\|u^x(t)-u^y(t)\|_V^2\Big]
    \le C_M \|x-y\|_V^2.
\end{equation}
\end{lemma}
\begin{proof}
The proof is inspired by \cite[\S 6]{BZZ26} and \cite[\S 4.5]{HLL}.
\\
Set $r := u^x-u^y$. Then, $r$ satisfies the following equation
\[
\begin{cases}
{\rm d} r(t) + \nu A r(t)\,{\rm d}t = \big(B(u^x(t))-B(u^y(t))\big)\,{\rm d}t + \big(g(u^x(t))-g(u^y(t))\big)\,{\rm d}w_0(t),\\
r(0) = y - x.
\end{cases}
\]
We apply the It\^o formula to $\|r\|_V^2$ and obtain, for every $t \ge0$, $\mathbb{P}$-a.s.
\begin{equation*}\begin{aligned}
    {\rm d}\|r(t)\|_V^2
    &+ 2\nu\|Ar(t)\|_H^2\, {\rm d}t
    = 2\big(Ar(t),B(u^x(t))-B(u^y(t))\big)_H\,{\rm d}t
    \\
&+ 2\big(r(t),g(u^x(t))-g(u^y(t))\big)_V\,{\rm d}w_0(t)
+ \big\|g(u^x(t))-g(u^y(t))\big\|_V^2\,{\rm d}t.
\end{aligned}\end{equation*}
We  estimate by means of the Cauchy and Young inequalities
\[
\big(Ar,B(u^x)-B(u^y)\big)_H
\le
\|Ar\|_H \|B(u^x)-B(u^y)\|_H
\le
\frac \nu 2 \|Ar\|^2_H +\frac 1{2\nu} \|B(u^x)-B(u^y)\|^2_H
\]
and using \eqref{stima-diff-B} we further obtain
\[\begin{split}
\|B(u^x)-B(u^y)\|^2_H
&\le
C\|u^x\|_{D(A)}\|u^x\|_V\|r\|_V^2
+ C \|r\|_{D(A)}\|r\|_V \|u^y\|_V^2.
\\
&\\
&\le C\|A u^x\|_H \|u^x\|_V \|r\|_V^2
+  \frac \nu 2\|Ar\|_H^2  +C_\nu  \|r\|^2_V \|u^y\|_V^4.
\end{split}\]
Inserting these estimates in the  It\^o's formula above, we obtain
\begin{equation*}\begin{aligned}
    {\rm d}\|r(t)\|_V^2
    &+ \nu\|Ar(t)\|_H^2\, {\rm d}t
    \le C_\nu (\|A u^x(t)\|_H \|u^x(t)\|_V \|r(t)\|_V^2
+  \|r(t)\|^2_V \|u^y(t)\|_V^4) \,{\rm d}t
    \\
    &\\
&+ 2\big(r(t),g(u^x(t))-g(u^y(t))\big)_V\,{\rm d}w_0(t)
+ \big\|g(u^x(t))-g(u^y(t))\big\|_V^2\,{\rm d}t.
\end{aligned}\end{equation*}
and neglecting the second term in the l.h.s. we get
\begin{equation*}\begin{aligned}
  {\rm d}\|r(t)\|_V^2
    &\le C_\nu (\|A u^x(t)\|_H \|u^x(t)\|_V \|r(t)\|_V^2
+  \|r(t)\|^2_V \|u^y(t)\|_V^4) \,{\rm d}t
+ \big\|g(u^x(t))-g(u^y(t))\big\|_V^2\,{\rm d}t
\\&\\
&+ 2\big(r(t),g(u^x(t))-g(u^y(t))\big)_V\,dw_0(t).
\end{aligned}\end{equation*}
Now we integrate on the time interval $[0,t\wedge \tau^{x,y}_M]$. Keeping in mind that
\[
\sup_{s\in [0,t\wedge \tau^{x,y}_M]}\|u^x(s)\|_V^p\le M^{p/2},
\qquad
\sup_{s\in [0,t\wedge \tau^{x,y}_M]}\|u^y(s)\|_V^p\le M^{p/2},
\]
\[
\int_0^{t\wedge \tau^{x,y}_M}\|Au^x(s)\|_H^2 {\rm d}s\le M,
\qquad
\int_0^{t\wedge \tau^{x,y}_M}\|Au^y(s)\|_H^2 {\rm d}s\le M,
\]
we estimate the deterministic integrals  in the r.h.s.  as follows:
\[\begin{split}
\int_0^{t\wedge \tau^{x,y}_M}
\|A u^x(s)\|_H \|u^x(s)\|_V \|r(s)\|_V^2 {\rm d}s
&\le
 M^{1/2}    \int_0^{t\wedge \tau^{x,y}_M}
\|A u^x(s)\|_H  \|r(s)\|_V^2 {\rm d}s
\\
&\le
 M^{1/2}   \left( \int_0^{t\wedge \tau^{x,y}_M}
\|A u^x(s)\|^2_H  {\rm d}s\right)^{1/2}
\left( \int_0^{t\wedge \tau^{x,y}_M} \|r(s)\|_V^4 {\rm d}s\right)^{1/2}
\\
&\le M  \left(\sup_{s\in [0,t\wedge \tau^{x,y}_M]} \|r(s)\|^2_V  \int_0^{t\wedge \tau^{x,y}_M} \|r(s)\|_V^2 {\rm d}s\right)^{1/2}
\\
&\le
M \varepsilon \sup_{s\in [0,t\wedge \tau^{x,y}_M]} \|r(s)\|^2_V
+
M C_{\varepsilon} \int_0^{t\wedge \tau^{x,y}_M} \|r(s)\|_V^2 {\rm d}s
\\
\intertext{and}
\int_0^{t\wedge \tau^{x,y}_M}
 \|r(s)\|^2_V \|u^y(s)\|_V^4
{\rm d}s
&\le
M^2 \int_0^{t\wedge \tau^{x,y}_M}
 \|r(s)\|^2_V {\rm d}s
\\
\intertext{and by the local Lipschitz assumption [A1a)]}
\int_0^{t\wedge \tau^{x,y}_M}
\|g(u^x(s))-g(u^y(s))\|_V^2
{\rm d}s
&\le
C_M^2 \int_0^{t\wedge \tau^{x,y}_M}
\|r(s)\|_V^2 {\rm d}s.
\end{split}\]
Therefore,  choosing $\varepsilon $ in a suitable way we get that  there exists  a positive constant $C_{M,\nu}$  such that
\begin{equation}\label{stima-senza-stochint}
\begin{aligned}
    \|r(t\wedge \tau^{x,y}_M)\|_V^2
    &\le
    \|r(0)\|_V^2 + \frac 14 \sup_{s\in [0,t\wedge \tau^{x,y}_M]} \|r(s)\|^2_V
    +C_{M,\nu} \int_0^{t\wedge \tau^{x,y}_M} \|r(s)\|_V^2\, {\rm d}s
\\
    &+2\int_0^{t\wedge \tau^{x,y}_M}
     \big(r(s),g(u^x(s))-g(u^y(s))\big)_V\,{\rm d}w_0(s).
\end{aligned}
\end{equation}

As far as the stochastic integral is concerned, by applying  the  Burkholder-Davis-Gundy inequality
we infer that 
\begin{align*}
&\E  \left|\sup_{t\in[0,T]}
\int_0^{t\wedge\tau_M^{x,y}}
\big(r(s),g(u^x(s))-g(u^y(s))\big)_V\,{\rm d}w_0(s)\right|
\\
&\le
\mathbb{E}\left(\int_0^{T\wedge\tau_M^{x,y}}\|r(s)\|_V^2\,\|g(u^x(s))-g(u^y(s))\|_V^2\,{\rm d}s\right)^{1/2}
\\
&\le \mathbb{E}\left(\int_0^{T\wedge\tau_M^{x,y}} C_M^2 \|r(s)\|_V^4\,{\rm d}s\right)^{1/2}
\quad \text{ by [A1a]}\\
& \le C_M  \mathbb{E}\left(
\sup_{s\in [0,T\wedge \tau^{x,y}_M]}  \|r(s)\|_V^2 \int_0^{T\wedge\tau_M^{x,y}}  \|r(s)\|_V^2\,{\rm d}s\right)^{1/2}
\\
& \le \tfrac 14  \mathbb{E}
\sup_{s\in [0,T\wedge \tau^{x,y}_M]}  \|r(s)\|_V^2
+  C \E \int_0^{T\wedge\tau_M^{x,y}}  \|r(s)\|_V^2\,{\rm d}s.
\end{align*}
Going back to \eqref{stima-senza-stochint}, we  consider first the supremum in time on $[0,T]$
 ($T$ finite and  fixed arbitrarily) and then the mathematical expectation in both sides;  we obtain
\[
\tfrac 12\E \sup_{t\in [0,T\wedge \tau^{x,y}_M]}
 \|r(t)\|_V^2
    \le
    \|r(0)\|_V^2
    +C\  \E \int_0^{T\wedge \tau^{x,y}_M} \|r(s)\|_V^2 {\rm d}s.
\]
From the stochastic Gronwall lemma (see \cite[Lemma 5.3]{GHZ}) we infer that there exists a constant $C_M$ (depending also on $\nu$) such that
\[
\mathbb{E}\sup_{t\in[0,T\wedge\tau_M^{x,y}]}\|u^x(t)-u^y(t)\|_V^2 \le C_M\,\|x-y\|_V^2.
\]
\end{proof}

\begin{lemma}\label{l2}
Consider a  $V$-converging sequence of initial velocities for the Navier-Stokes equation \eqref{SNS}:
$\displaystyle \lim_{n\to\infty}\|x_n-x\|_V=0$.  For  the sequence of stopping times
$\tau^{x_n,x} _M$ defined in \eqref{def-tau-n-M}, for every  finite time $T>0$
we have that
\begin{equation}\label{conv-stopping-times-n-M}
\lim_{M\to \infty} \sup_n \Po(\tau^{x_n,x}_M<T)=0.
\end{equation}
\end{lemma}
\begin{proof}
We start by noticing that, since $x_n\to x$ in $V$, the sequence $\|x_n\|_V$ is bounded; we set $R:=\sup_n\|x_n\|_V<\infty$.
Since, for every $z\in V$,
\begin{equation}\label{incl-eventi}
\{\tau_M^z<T\}
\subseteq
\left\{\sup_{s\in[0,T\wedge\tau^z_M]}\|u^z(s)\|_V^2\ge\tfrac M2\right\}
\cup
\left\{\int_0^{T\wedge\tau^z_M}\|Au^z(s)\|_H^2\,{\rm d}s\ge\tfrac M2\right\},
\end{equation}
and since $\tau^{x_n,x}_M=\min(\tau^{x_n}_M,\tau^x_M)$, we have
$\{\tau^{x_n,x}_M<T\}\subseteq\{\tau^x_M<T\}\cup\{\tau^{x_n}_M<T\}$; applying
\eqref{incl-eventi} to $z=x$ and $z=x_n$ separately, we obtain
\begin{multline}\label{stima-4pezzi}
\Po\left(\tau^{x_n,x}_M<T\right)
\le
\Po\left(\sup_{s\in[0,T\wedge\tau^x_M]}\|u^x(s)\|_V^2\ge\tfrac M2\right)
+\Po\left(\int_0^{T\wedge\tau_M^x}\|Au^x(s)\|_{H}^2\,{\rm d}s\ge\tfrac M2\right)
\\
+\Po\left(\sup_{s\in[0,T\wedge\tau^{x_n}_M]}\|u^{x_n}(s)\|_V^2\ge\tfrac M2\right)
+\Po\left(\int_0^{T\wedge\tau_M^{x_n}}\|Au^{x_n}(s)\|_{H}^2\,{\rm d}s\ge\tfrac M2\right).
\end{multline}
By Lemma \ref{est_in_P} applied with $z=x$ and
$z=x_n$, and since $\|x_n\|_V\le R$ for every $n$, each of the four terms on the
right-hand side of the above expression is bounded, uniformly in $n$, by an expression of the form
\[
\frac{\widehat C_T}{(M/2)^{p/2}},
\]
which does not depend on $n$ and tends to $0$ as $M\to\infty$, since $p>0$. Taking the supremum over $n\in\mathbb N$  in \eqref{stima-4pezzi} and then
letting $M\to\infty$, we conclude that
\[
\lim_{M\to\infty}\sup_n\Po(\tau^{x_n,x}_M<T)=0,
\]
which concludes the proof.
\end{proof}

We have now all the ingredients to prove Theorem \ref{teo-Feller}.

\begin{proof}[Proof of Theorem \ref{teo-Feller}]
The proof of the existence of an invariant measure is inspired by \cite[\S 4.5]{HLL}.
We write it here since have a different definition of stopping time; but as the reader can see, it does not affect this proof.

Let $t>0$ and $\psi \in \C_b(V)$ be fixed. To prove that the semigroup $P$ defined in \eqref{def-Pt} is Feller,
we have to prove that, given a sequence $(x_n )_n \subset V$
which converges in $V$ to
$x\in V$ as $n \rightarrow \infty$,  the sequence $P_t\psi(x_n)$ converges to
$P_t\psi(x) $ as $n \rightarrow \infty$.
    We notice that, since $\psi \in \C_b(V)$, it is enough  to show that
    \begin{equation}\label{in-probab}
    u^{x_n}(t)\to u^x(t) \text{ in probability} .
    \end{equation}
    In fact, since $\psi$ is continuous, the above implies that $\psi(u^{x_n}(t))\to \psi(u^{x}(t))$
    in probability. Moreover, since $\psi$ is bounded, from the dominated convergence theorem
    it follows
    \[
    \E[\psi(u^{x_n}(t))] \to \E[\psi(u^{x}(t))],
    \]
    which is exactly what we need to prove.

Let us thus prove \eqref{in-probab}. We consider the stopping time
   $\tau^{x_n,x} _M$ introduced in \eqref{def-tau-n-M}.
Exploiting the Chebyshev inequality, thanks to Lemma \ref{l1}, we infer, for every $\eps>0$
\[\begin{split}
\Po\left(\sup_{t \in [0,T]} \|u^{x_n}(t)-u^x(t)\|_V>\eps\right)
&\le
\Po\left(\sup_{t \in [0,T]} \|u^{x_n}(t)-u^x(t)\|_V>\eps, T\le \tau^{x_n,x}_M\right)
+
\Po\left(\tau^{x_n,x}_M<T\right)
\\&
\le \frac 1{\eps^2}\E \left[\sup_{t \in [0,T\wedge \tau^{x_n,x}_M]} \|u^{x_n}(t)-u^x(t)\|_V^2\right]+
\Po\left(\tau^{x_n,x}_M<T\right)
\\&\le
\frac{C_M}{\eps^2}\|x_n-x\|_V^2+
\Po\left(\tau^{x_n,x}_M<T\right).
\end{split}\]
Letting  $n\to\infty$
we get
\[
0\le \limsup_{n\to\infty} \Po\left(\sup_{t \in [0,T]} \|u^{x_n}(t)-u^x(t)\|_V>\eps\right)
\le \limsup_{n\to\infty} \Po\left(\tau^{x_n,x}_M<T\right) \le  \sup_n \Po\left(\tau^{x_n,x}_M<T\right).
\]
Now, letting $M\to \infty $, thanks to Lemma \ref{l2}, we get
\[
\lim_{n\to\infty}\Po\left(\sup_{t \in [0,T]}  \|u^{x_n}(t)-u^x(t)\|_V>\eps\right) =0,
\]
which proves the  convergence in probability and concludes the proof.
\end{proof}

\subsection{Proof of the existence of an invariant measure}
\label{measures}

\begin{proof}[Proof of part (i) of Theorem \ref{thm_existence_invmeas}]
The proof for the existence of an invariant measure relies on the Krylov-Bogoliubov technique (see, e.g., \cite[\S 4.5]{DPZ1}) that requires that the transition semigroup $P$ is Feller and the tightness property for time-averaged laws.
\\
The Feller property has been proved in Theorem \ref{teo-Feller}. It remains to prove the tightness property.
To this end, let  the initial velocity $x \in V$ vanish  and denote by  $u^{0}$  the corresponding solution of problem \eqref{SNS}. We have to show that the family of measures
$(\mu_t)_{t>0}\subset \mathscr{P}(V)$ defined by
\begin{equation*}
\mu_t : A \mapsto \frac 1t \int_0^t (P_s\pmb{1}_{A})(0)\, {\rm d}s
=\frac 1t \int_0^t \mathbb{P}\left(u^0(s) \in A \right)\, {\rm d}s, \qquad A \in \cB(V), \ t>0,
\end{equation*}
is tight.
Let $B_n$ be the closed ball in $D(A)$ of radius $n \in \mathbb{N}$.
Then $B_n$ is a compact subset of $V$ since the embedding $D(A) \hookrightarrow V$ is compact.
Hence, the Chebychev inequality and Proposition \ref{est_H2_norm} yield, for every $t\ge 1$,
\begin{equation*}\begin{split}
\mu_t(B_n^c)
&=\frac 1t \int_0^t \mathbb{P}\left(u^0(s) \in A \right)\, {\rm d}s
\le
\frac 1t \int_0^t \mathbb{P} \big(\|Au^0(s)\|_H  >  \tfrac 1 {\sqrt{1+\lambda_1^{-2}}} n \big)\, {\rm d}s
\\
&\le \frac{C}{n^p}\mathbb{E} \left[ \frac{1}{t}\int_0^t \|Au^0(s)\|^p_H\, {\rm d}s\right]
\le \frac{C'}{n^p}.
\end{split}\end{equation*}
Therefore
\[
\forall \eps>0 \;\quad \exists\ n_\eps:
\
\sup_{t\ge 1}\mu_t\left(B_{n_\varepsilon}^c\right)<\eps
\]
and the thesis follows.

\end{proof}

\subsection{Proof of properties of the invariant measures}
\label{measures-prop}

\begin{proof}[Proof of part (ii) of Theorem \ref{thm_existence_invmeas}]
    Let $\mu \in \mathscr{P}(V)$ be an invariant measure
for the transition semigroup $P$.
\\
{\sc Step 1.} Let us start by showing that there exists for a positive constant $C$, independent of $\mu$,  such that 
\begin{equation}\label{sti_V}
\int_{V} \|x\|^p_V\, \mu({\rm d} x) \leq C.
\end{equation}
To this end, we consider the mapping
$\Phi$ defined as
\begin{equation*}
\Phi: V \ni  x \mapsto \|x\|^p_{V} \in [0,+\infty] 
\end{equation*}
and its approximations $\{\Phi_n\}_{n\in\mathbb{N}}$, where, for every $n\in\mathbb{N}$,
$\Phi_n:V\to[0,n^p]$ is defined (setting $B_n^V$ as the closed ball of radius $n$ in $V$) as
\begin{equation*}
\Phi_n:V \ni x \mapsto
\begin{cases}
\|x\|^p_{V} & \text{if} \ x \in B_n^V,
\\
n^p & \text{otherwise}.
\end{cases}
\end{equation*}
It is not difficult to check that actually $\Phi_n\in\C_b(V)$ for every $n\in\mathbb{N}$.
Exploiting the invariance of $\mu$ and the definition of $P$, we have
\begin{equation}\label{invariance-step1}
\int_{V} \Phi_n(x)\, \mu({\rm d} x)=
\int_{V} P_s \Phi_n(x)\, \mu({\rm d} x)=  \int_{V} \E\left[\Phi_n(u^x(s))\right]\, \mu({\rm d} x)
\qquad \forall\, s\ge 0.
\end{equation}
Since $\Phi_n$ is bounded by $n^p$, 
we can apply the Fatou lemma for the limsup to the r.h.s. of
\eqref{invariance-step1} as $s\to\infty$, obtaining
\[
\int_{V} \Phi_n(x)\, \mu({\rm d} x)
=\limsup_{s\to\infty}\int_{V} \E\left[\Phi_n(u^x(s))\right]\, \mu({\rm d} x)
\le \int_{V} \limsup_{s\to\infty}\E\left[\Phi_n(u^x(s))\right]\, \mu({\rm d} x).
\]
Since $\Phi_n(y)\le\|y\|_V^p$ pointwise for every $y\in V$, we have
$\E[\Phi_n(u^x(s))]\le\E[\|u^x(s)\|_V^p]$ for every $x\in V$ and $s\ge0$, and therefore,
by Proposition \ref{est_V_prop}, we infer
\[
\limsup_{s\to\infty}\E\left[\Phi_n(u^x(s))\right]
\le \limsup_{s\to\infty}\mathbb{E}\left[\|u^x(s)\|^p_V\right]\le \bar{C_1}
\qquad \forall\, x\in V.
\]
We conclude that,
for every $n \in \mathbb{N}$,
\[
\int_{V} \Phi_n(x)\, \mu({\rm d} x) \le  \bar{C_1}.
\]
We point out that  the constant $ \bar{C_1}$ is independent of $\mu$.
Since $\Phi_n$ converges pointwise and monotonically from below to $\Phi$,
the Monotone Convergence Theorem yields \eqref{sti_V}, with $C:= \bar{C_1}$.

{\sc Step 2.} We prove now that
\begin{equation}\label{sti_Z}
\int_{V} \|x\|^p_{D(A)}\, \mu({\rm d} x) \le C.
\end{equation}
To this end, we argue as in the previous step, considering the map
$\Psi:V \rightarrow [0,+\infty]$ defined by
\begin{equation*}
\Psi: V \ni  x \mapsto \|x\|^p_{D(A)}\pmb{1}_{V \cap D(A)}(x)+ \infty \ \pmb{1}_{V \cap D(A)^C}(x) \in [0,\infty),
\end{equation*}
and its approximations $\{\Psi_n\}_{n\in\mathbb{N}}$
defined, with   $B_n^{D(A)}$ denoting  the closed ball of radius $n$ in $D(A)$,  by 
\begin{equation*}
\Psi_n:V \ni x \mapsto
\begin{cases}
\|x\|^p_{D(A)} & \text{if} \ x \in B_n^{D(A)} \cap V,
\\
n^p & \text{otherwise},
\end{cases}
\in [0,n^p].
\end{equation*}
One has that $\Psi_n\in\mathcal{B}_b(V)$ for every $n\in\mathbb{N}$.
Hence, arguing as above by using
the invariance of $\mu$ (keeping in mind Remark \ref{rem_inv}), the boundedness of $\Psi_n$, the definition of $P$,
and the Fubini-Tonelli Theorem, we get
\begin{align*}
\int_{V} \Psi_n(x)\, \mu({\rm d} x)
&=
\int_{V}\int_0^1 \mathbb{E}\left[\Psi_n(u^x(s)) \right]\,{\rm d}s\, \mu({\rm d} x)
\\
&\le \int_{ V}\int_0^1 \mathbb{E}\left[\|u^x(s)\|_{D(A)}^p \right]\,{\rm d}s\, \mu({\rm d} x)
\le (1+\lambda_1^{-2})^{p/2} \int_{ V}\int_0^1 \mathbb{E}\left[\|Au^x(s)\|_{H}^p \right]\,{\rm d}s\, \mu({\rm d} x).
\end{align*}
Proposition \ref{est_H2_norm} together with the estimate \eqref{sti_V} entail then
\[
  \int_{V} \Psi_n(x)\, \mu({\rm d} x) \le
  C\left(1+\int_{V}\|x\|_V^p\,\mu({\rm d} x)\right) \leq C.
\]
Applying then the Monotone Convergence Theorem implies  \eqref{sti_Z}.
\end{proof}

\section{Existence of an ergodic invariant measure}
\label{sect:ergmeas}

For every invariant measure $\mu$,
by density and by definition of invariance,
the transition  semigroup $P$ can be extended (with the same symbol for brevity) to
a strongly continuous linear semigroup of contractions on $L^q(V,\mu)$ for every $q\in[1,+\infty)$.
 Here we consider its extension in  $L^2(V,\mu)$.

One possible definition of ergodic measure is the following (see \cite[Theorem 3.2.4]{DPZ2}).

\begin{definition}
  \label{def:erg}
  An invariant measure $\mu\in\mathscr{P}(V)$ for the  transition semigroup $P$ is said to be ergodic if
  \[
  \lim_{t \rightarrow \infty}\frac 1t \int_0^tP_s \varphi\, {\rm d}s
  = \int_V \varphi \, {\rm d} \mu \qquad \text{in } L^2(V, \mu)\quad \forall\, \varphi \in L^2(V, \mu).
  \]
\end{definition}

It is well known that if there exists a unique invariant measure, then it is ergodic.
So far we have proved  the existence but not the uniqueness of  the invariant measures of the stochastic
Navier-Stokes equation \eqref{SNS}.
Anyway, we can prove that there exists at least one ergodic invariant measure.
\begin{theorem}
\label{ergodic}
Assume [A1], [A2] and [A3]. Then,
there exists an ergodic invariant measure for the transition semigroup $P$.
\end{theorem}

\begin{proof}
It is well known (see, e.g., \cite{Ali}[Theorem 19.25]) that,
 for an arbitrary Markov transition semigroup $\{P_t\}_{t\ge 0}$,
the ergodic measures are precisely the extreme
points of the convex set of its invariant measures. On the other hand, the Krein--Milman
Theorem, see, e.g., \cite{Ali}[Theorem 7.68], characterizes the convex compact subsets of a
locally convex Hausdorff space as the closed convex hull of their extreme points.

Let us denote by $\Gamma\subset \mathscr{P}(V)$ the convex set of all invariant measures for the
Markov semigroup $\{P_t\}_{t\ge0}$ obtained in Theorem \ref{thm_existence_invmeas}. There we proved that
$\Gamma$ is non empty; thus, in view of the above discussion, it only remains to show that
$\bar\Gamma$ (the closure of $\Gamma$ in $\mathscr{P}(V)$, with respect to the topology of
weak convergence) is compact, or, equivalently by Prokhorov's theorem, that $\Gamma$ is tight.

By estimate \eqref{int_mu_V} in Theorem \ref{thm_existence_invmeas}, there exists a constant $C>0$, depending
on the structural parameters, such that
\begin{equation*}
\int_{V}\|x\|^p_{D(A)}\, \mu({\rm d}x) \le C \qquad \forall\,\mu\in \Gamma,
\end{equation*}
By the Chebychev inequality we infer
\[
\mu(x\in V: \|x\|_{D(A)}>R)\le \frac {\int_{V}\|x\|^p_{D(A)}\mu({\rm d}x)} {R^p}
\le \frac C  {R^p}
\]
with a bound which is uniform w.r.t. any $\mu\in\Gamma$,
since the constant $C$ does not depend on $\mu$.
Bearing in mind the compactness of the embedding $D(A)\subset V$, we obtain
that  the set $\Gamma$ is tight.
The Krein--Milman theorem then applies and yields that $\overline\Gamma$ admits
(at least one) extreme point. By the characterization recalled at the beginning of the
proof, such an extreme point is an ergodic invariant measure for $\{P_t\}_{t\ge0}$, which
concludes the proof.
\end{proof}

\appendix

\section{Estimates on the solution}

\subsection{Estimates in expectation}
We prove some a priori estimates for the  global unique solution $u$ of equation \eqref{SNS}.
For this purpose we will use the following functional. Fix $p >0$ and define
\[
\cE : V \to [1,+\infty), \qquad u \mapsto \cE(u) := \big(1 + \|u\|_V^2\big)^{p/2}.
\]
 The first and second derivatives of $\cE$ are given by
\[
\cE^\prime(u)[h]=  p (1+\| u\|_V^2)^{\frac p2-1}(u, h)_V
                                  = p \frac{(u, h)_V}{(1+\| u\|_V^2)^{1-\frac p2}}, \qquad u, h \in V
\]
\begin{align*}
\cE^{\prime\prime}(u)[h,k]
&=p(\tfrac p2-1)(1+\|u\|_V^2)^{\frac p2- 2}2 ( u, k)_V ( u, h)_V
+p(1+\|u\|_V^2)^{\frac p2-1}(k, h)_V
\\
&=- p(2-p) \frac{(u, k)_V (u,h)_V}{(1+\|u\|_V^2)^{2-\frac p2}}
  + p\frac{( k, h)_V}{(1+\| u\|_V^2)^{1-\frac p2}}, \qquad u, h,k \in V.
\end{align*}

\begin{proposition}
\label{est_V_prop}
Let assumptions [A1], [A2] and [A3] be in force and let $u_0 \in V$.
for every
$p \in (0, 1)$  there exists a positive constant $\bar C_1$, depending on the structural parameters, such that
\[
\E \left[\| u(t)\|_V^p\right]
\le
(1+\|u_0\|_V^2)^{\frac p2}\,e^{-t} +\bar C_1,
\]
for every $t\ge 0$.
\end{proposition}
\begin{proof}
We apply the
It\^o formula to  $\|u\|_V^2=\langle A u,u\rangle $ and obtain, for every $t \ge 0$, $\mathbb{P}$-a.s.
\begin{multline}
{\rm d}\| u(t)\|_V^2
=
- 2\nu \|A u(t)\|_H^2{\rm d}t
+ 2\big( f,Au(t)\big)_H\,{\rm d}t
- 2\big(  B(u(t)),Au(t) \big)_H {\rm d}t
+ 2\varphi(u(t)) \| u(t)\|_V^2 {\rm d}w_0(t)
\\
+2 \sum_{j=1}^\infty \gamma_j \big( u(t),e^V _j\big)_V {\rm d}w_j(t)
+ \left(\varphi(u(t))^2 \|u(t)\|_V^2
+ \sum_{j=1}^\infty \gamma_j^2 \right){\rm d}t.
\end{multline}
The Cauchy-Schwartz and Young inequalities and estimate \eqref{prima-stime-B} yield
\[
|\big( B(u),Au\big)_H |
\le
C\|u\|_V^{\frac 32} \|u\|_{D(A)}^{\frac 32}
\le
\frac \nu 4 \|Au\|_H^2 +\frac {C_\nu}2 \|u\|_V^6,
\]
while, by the Cauchy--Schwarz and Young inequalities,
\[
\big(  f,Au\big)_H
\ \le\ \|f\|_H\|Au\|_H
\ \le\ \frac\nu4\|Au\|_H^2+\frac4\nu\|f\|_H^2,
\]
from which we infer
\begin{multline}\label{Ito_V}
{\rm d}\| u(t)\|_V^2
+\nu \|A u(t)\|_H^2{\rm d}t
\le
\left( C_\nu \|u(t)\|_V^6 + \varphi(u(t))^2 \|u(t)\|_V^2
+\sum_{j=1}^\infty\gamma_j^2
+\frac2\nu\|f\|_H^2
\right){\rm d}t
\\
+2\varphi(u(t)) \| u(t)\|_V^2 {\rm d}w_0(t)
+2\sum_{j=1}^\infty \gamma_j \big( u(t),e^V _j\big)_V {\rm d}w_j(t).
\end{multline}
We now apply the It\^o formula to $\cE(u)=(1+\|u\|_V^2)^{p/2}$
with $p>0$. Later on a further restriction on $p$ will appear.
Bearing in mind \eqref{Ito_V}, we infer
\begin{equation}\begin{aligned}
(1+\| u(t)\|_V^2)^{\frac p2}
&+ \frac p 2 \nu \int_0^t \frac{\|A u(s)\|_H^2}{(1+\| u(s)\|_V^2)^{1-\frac p2}}{\rm d}s
\\
&\le
(1+\|u_0\|_V^2)^{\frac p2}
+\frac p 2 C_\nu \int_0^t \frac{\| u(s)\|^6_V}{(1+\| u(s)\|_V^2)^{1-\frac p2}}{\rm d}s
+ M(t)
\\
&-\frac p2 (2-p)\int_0^t \frac{\|u(s)\|_V^4 \varphi^2(u(s))}{(1+\|u(s)\|_V^2)^{2-\frac p2}} {\rm d}s
+\frac p2 \int_0^t \frac {\|u(s)\|_V^2 \varphi^2(u(s))}{(1+\|u(s)\|_V^2)^{1-\frac p2}} {\rm d}s
\\
&-\frac p2 (2-p)\int_0^t \frac{\sum_j \gamma_j^2 (u(s), e^V_j)_V ^2}{(1+\|u(s)\|_V^2)^{2-\frac p2}} {\rm d}s
+\frac p2 \int_0^t \frac {\sum_j \gamma_j^2 +\;\frac2\nu \|f\|_H^2}{(1+\|u(s)\|_V^2)^{1-\frac p2}} {\rm d}s,
\end{aligned}
\end{equation}
with the martingale term
\[
M(t):=
p \int_0^t  \frac{\varphi(u(s)) \|u(s)\|_V^2 } {(1+\| u(s)\|_V^2)^{1-\frac p2}}\,{\rm d}w_0(s)
+ p\sum_{j=1}^\infty \gamma_j   \int_0^t \frac{(u(s),  e_j^V )_V }{(1+\| u(s)\|_V^2)^{1-\frac p2}}\,{\rm d}w_j(s).
\]
When $p\in (0,2)$,  we bound $\frac 1 {(1+\| u\|_V^2)^{1-\frac p2}} $ by 1.
Moreover,  neglecting one negative addend and summing up three integrals we get
\begin{equation}
\label{stimaxtight}
\begin{aligned}
&(1+\| u(t)\|_V^2)^{\frac p2}
+ \frac p 2 \nu \int_0^t \frac{\|A u(s)\|_H^2}{(1+\| u(s)\|_V^2)^{1-\frac p2}}{\rm d}s
\\
&\le
(1+\|u_0\|_V^2)^{\frac p2}
+ M(t) +\frac p2  \left(\sum_j \gamma_j^2 +\;\frac2\nu\|f\|_H^2\right)  t
\\
&+ \frac p 2  \int_0^t
 \frac{ C_\nu  \| u(s)\|^6_V (1+\| u(s)\|_V^2)   - (2-p) \varphi^2 (u(s)) \|u(s)\|_V^4 +  \varphi^2(u(s)) \|u(s)\|_V^2(1+\|u(s)\|_V^2)}{(1+\|u\|_V^2)^{2-p/2}}
         {\rm d}s .
         \end{aligned}
\end{equation}
Now we consider the numerator of the latter integrand; using  Lemma  \ref{lemA}  when $0<p<1$ we get
\[\begin{split}
-(2-p) \varphi^2(u) \|u\|_V^4 + \varphi^2(u) \|u\|_V^2(1+\|u\|_V^2)
&=
  \varphi^2(u) \big(\|u\|_V^2 - (1-p)\|u\|_V^4\big)
\\&
= \varphi^2(u) \big(\|u\|_V^2 - \tfrac{1-p}2 \|u\|_V^4\big)  -  \varphi^2(u)\tfrac{1-p}2 \|u\|_V^4
\\&
\le K_{\frac{1-p}2} \|u\|_V^2- \varphi^2(u)\tfrac{1-p}2 \|u\|_V^4.
\end{split}\]
Using this estimate and the assumption [A1b)], we can  bound the integrand in the r.h.s. of \eqref{stimaxtight}  by
\begin{equation}\label{frazione-3pezzisopra}
\frac{p}{2} \times \frac{C_\nu \|u\|_V^6(1+\|u\|_V^2) +
      K_{\frac{1-p}2} \|u\|_V^2- \tfrac{1-p}2  b^2 \|u\|_V^{2(\kappa+2)}}{(1+\|u\|_V^2)^{2-p/2}}.
\end{equation}
Since $\kappa>2, $
we can find a  positive constant
$\widetilde C$, depending on $C_\nu$, $K_{\frac{1-p}2}$, $p$, $b$ and $\kappa$,    such that 
\[
C_\nu \|u\|_V^6(1+\|u\|_V^2) +
      K_{\frac{1-p}2} \|u\|_V^2- \tfrac{1-p}2  b^2 \|u\|_V^{2(\kappa+2)}
      \le
      \widetilde C - \tfrac 2p(1+\|u\|_V^2)^{2},
\]
for every $u\in V$. In fact this corresponds to the inequality
\[C_\nu y^6(1+y^2) +
      K_{\frac{1-p}2} y^2+\tfrac 2p(1+y^2)^{2}
      \le  \tfrac{1-p}2  b^2 y^{2(\kappa+2)}
      +
      \widetilde C.\]
Substituting in \eqref{stimaxtight} we get the estimate
\begin{equation}\label{stimaxtight3}
    \begin{aligned}
(1+\| u(t)\|_V^2)^{\frac p2}
+ \frac p 2 \nu \int_0^t \frac{\|A u(s)\|_H^2}{(1+\| u(s)\|_V^2)^{1-\frac p2}}{\rm d}s
\le
(1+\|u_0\|_V^2)^{\frac p2} - \int_0^t (1+\|u(s)\|_V^2)^{\frac p2}  {\rm d}s
\\
+ M(t) +\frac p2  \left(\widetilde C+ \sum_j \gamma_j^2 +\;\frac2\nu\|f\|_H^2\right)  t.
\end{aligned}
\end{equation}
Now we neglect the integral in the l.h.s.
Taking the expected value on both sides, thanks to the Gronwall's lemma we infer
\[
\mathbb{E}\left[(1+\| u(t)\|_V^2)^{\frac p2}\right] \le (1+\|u_0\|_V^2)^{\frac p2}\,e^{-t}
 +\frac p2  \left(\widetilde C+ \sum_j \gamma_j^2 +\;\frac2\nu\|f\|_H^2\right) .
\]
and the thesis follows as an immediate consequence.
\end{proof}

The second result involves the $D(A)$-norm of the solution.

\begin{proposition}
\label{est_H2_norm}
Let assumptions [A1], [A2] and [A3] be in force and let $u_0\in V$. for every
$p \in (0, 1)$ there exists a positive constant $\bar C_2$, depending on the structural parameters,  such that
\begin{equation}\label{stima-in-media-uniforme}
\E \left[\int_0^t  \|A u(s)\|_{H}^{p} {\rm d }s\right]
    \le
    \bar C_2 \left[(1+\| u_0\|_V^2)^{\frac p2} +t\right]
\qquad\forall \ t >0.
\end{equation}
\end{proposition}
\begin{proof} We go back to the proof of the previous Proposition.
From \eqref{stimaxtight3} we get  that
\begin{equation}\label{la-stima-per-D(A)}
   \frac p2  \nu\E \int_0^t  \frac{\|A u(s)\|_H^2}{(1+\| u(s)\|_V^2)^{1-\frac p2}} {\rm d }s
    \le
    (1+\| u_0\|_V^2)^{\frac p2} +\frac p2 \left(\widetilde C+ \sum_j \gamma_j^2 +\;\frac2\nu\|f\|_H^2\right) t.
\end{equation}
The Poincar\'e inequality \eqref{Poincare-in-DA} yields
\begin{equation}\label{semplice-poincare}
\frac 1{(\lambda_1+\|Au\|_H^2)^{1-\frac{p}2}}
\le
\frac 1{(\lambda_1+\lambda_1\|u\|_V^2)^{1-\frac{p}2}}.
\end{equation}
Using \eqref{semplice-poincare}, we bound $\|Au\|_H^p$ as follows
\[
\|Au\|_H^p
= \|Au\|_H^p\frac {(\lambda_1+\|Au\|^2_H)^{1-\frac p2} } {(\lambda_1+\|Au\|_H^2)^{1-\frac p2}}
\le \|Au\|_H^p
\frac {(\lambda_1+\|Au\|^2_H)^{1-\frac p2} } {(\lambda_1+\lambda_1\|u\|_V^2)^{1-\frac{p}2}}
=
\frac 1 {\lambda_1^{1-\frac{p}2}} \frac {(\lambda_1\|Au\|^{\frac{2p}{2-p}}_H +\|Au\|^{\frac 4{2-p}}_H)^{\frac {2-p}2}} {(1+\|u\|_V^2)^{1-\frac{p}2}}.
\]
By the Young inequality, $\lambda_1^{1-\frac p2}\|Au\|_H^{p}\le\|Au\|_H^{2}+C$ for a constant
$C:=C(p,\lambda_1)>0$; hence the numerator is bounded by $C+2\|Au\|_H^{2}$. Therefore,
\[
\|Au\|_H^p
\le
\frac 1 {\lambda_1^{1-\frac{p}2}} \frac {C +2\|Au\|_H^{2} } {(1+\|u\|_V^2)^{1-\frac{p}2}}
\le
\frac C{\lambda_1^{1-\frac{p}2}}+\frac 2{\lambda_1^{1-\frac{p}2}} \frac{\|Au\|_H^{2}}{(1+\|u\|_V^2)^{1-\frac{p}2}}.
\]
Integrating in time, taking the mathematical expectation on both sides and
keeping in mind \eqref{la-stima-per-D(A)},  we conclude the proof.
\end{proof}

\subsection{Localized estimates in probability}
In this Section we reformulate and prove \cite[Lemma 7.7]{BZZ26} with the addition of additive noise and deterministic forcing.

\begin{lemma}
\label{est_in_P}
Let $\tau^x_M$ be the stopping time defined in \eqref{tau_M}. Let $p \in (0,1)$.
Then, for every
$R\ge1$ and $T\ge0$, there exist constants $\widehat C_T^1,\widehat C_T^2>0$,
depending also on $p, u_0, \sum \gamma_j^2,\|f\|_H^2$ and the structural parameters, but not
on $R$, such that
\begin{equation}\label{tauM-per-V}
\Po\Big(\sup_{s\in[0,T\wedge\tau^x_M]}\|u(s)\|_V^2\ge R^2\Big)
\ \le\ \frac{\widehat C_T^1}{R^p},
\end{equation}
\begin{equation}\label{tauM-per-D(A)}
\Po\Big(\int_0^{T\wedge\tau_M}\|Au^x(s)\|_H^2\,{\rm d}s\ge R^2\Big)
\ \le\ \frac{\widehat C_T^2}{R^p}.
\end{equation}
\end{lemma}

\begin{proof}
Starting from \eqref{stimaxtight3} we obtain, for every $t\ge0$,
\[
(1+\|u(t)\|_V^2)^{\frac p2}
+\frac p2\nu\int_0^t\frac{\|Au(s)\|_H^2}{(1+\|u(s)\|_V^2)^{1-\frac p2}}\,{\rm d}s
\ \le\ (1+\|u_0\|_V^2)^{\frac p2}+M(t)+C't,
\]
for a constant $C'>0$ depending on the structural parameters, $\sum\gamma_j^2$ and
$\|f\|_H^2$. Taking the integral from $0$ to $T \wedge \tau^x_M$ and
then expectation, for every $T\ge0$, we infer
\begin{equation}\label{eq_old_proof}
\E\big[\cE(u(T\wedge\tau^x_M))\big]
+\frac p2\nu\,\E\int_0^{T\wedge\tau^x_M}\frac{\|Au(s)\|_H^2}{(1+\|u(s)\|_V^2)^{1-\frac p2}}\,{\rm d}s
\ \le\ \cE(u_0)+C'T,
\end{equation}
where we recall that $\cE(u):=(1+\|u\|_V^2)^{p/2}$.

Let us now prove \eqref{tauM-per-V}. Define
\[
\sigma_R:=\inf\{t\in[0,\tau_M^x):\|u(t)\|_V\ge R\}.
\]
It holds that
\begin{equation}\label{est_cont}
\Po\left(\sup_{t\in[0,T\wedge\tau_M^x]}\|u(t)\|_V\ge R\right)\ \le\ \Po(\sigma_R\le T\wedge\tau_M^x).
\end{equation}
Since $\|u(\sigma_R)\|_V=R$ on $\{\sigma_R\le T\wedge\tau_M^x\}$, we infer that
\[
\cE(u(\sigma_R))=(1+R^2)^{p/2}\ge R^p,
\]
which implies
\[
\cE\big(u(T\wedge\sigma_R\wedge\tau_M^x)\big)\ \ge\ R^p\,\mathbf 1_{\{\sigma_R\le T\wedge\tau_M^x\}}.
\]
Hence, in view of \eqref{eq_old_proof}, there exists $\widehat C_T^1>0$ such
that
\begin{equation}\label{est_cont2}
\Po(\sigma_R\le T\wedge\tau_M^x)
\ \le\ \frac1{R^p}\,\E\big[\cE(u(T\wedge\sigma_R\wedge\tau_M^x))\big]
\ \le\ \frac1{R^p}\big(\cE(u_0)+C'T\big)=:\frac{\widehat C_T^1}{R^p}.
\end{equation}
Therefore, \eqref{est_cont} and \eqref{est_cont2}  yields
\eqref{tauM-per-V}.

Let us now deal with \eqref{tauM-per-D(A)}. It holds
\begin{equation*}\begin{aligned}
\Po\left(\int_0^{T\wedge\tau_M^x}\|Au(t)\|_H^2\,{\rm d}t\ge R^2\right)
&\le\ \Po\left(\sup_{t\in[0,T\wedge\tau_M^x]}\|u(t)\|_V\ge R\right)
\\
&+\Po\left(\int_0^{T\wedge\tau_M^x}\|Au(t)\|_H^2\,{\rm d}t\ge R^2,\
\sup_{t\in[0,T\wedge\tau_M^x]}\|u(t)\|_V<R\right).
\end{aligned}\end{equation*}
Note that if $\sup_{t\in[0,T\wedge\tau_M^x]}\|u(t)\|_V<R$, then, since $p \in (0,1)$,
\[
(1+\|u(t)\|_V^2)^{\frac{p-2}2}\ \ge\ (1+R^2)^{\frac{p-2}2},
\]
which implies
\[
\int_0^{T\wedge\tau_M^x}\|Au(t)\|_H^2\,{\rm d}t
\ \le\ (1+R^2)^{\frac{2-p}2}\int_0^{T\wedge\tau_M^x}(1+\|u(t)\|_V^2)^{\frac{p-2}2}\|Au(t)\|_H^2\,{\rm d}t.
\]
Keeping in mind that $R \ge 1$, thanks to estimate \eqref{tauM-per-V}, by the Markov inequality and estimate \eqref{eq_old_proof}, there exists $\widehat C_T^2>0$ such that \
\[
\Po\left(\int_0^{T\wedge\tau_M^x}\|Au(t)\|_H^2\,{\rm d}t\ge R^2\right)
\ \le\ \frac{\widehat C_T^1}{R^p}
+\Po\left(\int_0^{T\wedge\tau_M^x}(1+\|u(t)\|_V^2)^{\frac{p-2}2}\|Au(t)\|_H^2\,{\rm d}t
\ge\frac{R^2}{(1+R^2)^{\frac{2-p}2}}\right)
\]
\[
\le\ \frac{\widehat C_T^1}{R^p}
+\frac{2\big(\cE(u_0)+C'T\big)}{p\nu}\frac{(1+R^2)^{\frac{2-p}2}}{R^2}
\ \le\ \frac{\widehat C_T^1}{R^p}
+\frac{2^{2-\frac p2}}{p\nu}\frac{\cE(u_0)+C'T}{R^p}
=:\frac{\widehat C_T^2}{R^p}.
\]
This concludes the proof.
\end{proof}

\begin{remark}
\label{rem_k_2}
Propositions \ref{est_V_prop}, \ref{est_H2_norm} and Lemma \ref{est_in_P} also hold if $\kappa=2$ in Assumption [A1b)], provided
that $b$ is sufficiently large. Indeed, if $|\varphi(u)|\ge b\|u\|_V^2$, then, arguing as
above, the integrand in \eqref{frazione-3pezzisopra} is bounded by
\[
\tfrac{p}{2}\,\frac{C_\nu\|u\|_V^6(1+\|u\|_V^2)+K_{\frac{1-p}2}\|u\|_V^2-\tfrac{1-p}{2}\,b^2\|u\|_V^{8}}
{(1+\|u\|_V^2)^{2-p/2}}.
\]
Now $C_\nu\|u\|_V^6(1+\|u\|_V^2)$ contains the term $C_\nu\|u\|_V^8$, of the same degree as
the negative term. If
\begin{equation}\label{cond-kappa2}
C_\nu<\tfrac{1-p}{2}\,b^2 ,
\end{equation}
the coefficient of $\|u\|_V^8$ in the numerator is negative, so the numerator is bounded
above by $\widetilde C-\tfrac2p(1+\|u\|_V^2)^2$ for a suitable $\widetilde C>0$, exactly as in
the case $\kappa>2$, and the rest of the proofs are unchanged. Note also that
\eqref{cond-kappa2} is a restriction on $p$ for a given $b$: it holds if and only if
$p<1-2C_\nu/b^2$.
\end{remark}

\section{Sketch of the proof of Theorem \ref{ex_uniq_sol}}
\label{proof_well_pos}

\begin{proof}[Proof of Theorem \ref{ex_uniq_sol}]

We generalize  the proof  done in \cite{BZZ26}  for the Navier-Stokes equation \eqref{SNS} with $G=0$ and $f=0$ to the case  with also an additive noise and a  deterministic forcing term. We can assume $f\in L^2_{\text{loc}}(0,\infty;H)$, even if we considered before that $f$ is time independent, in order to get the existence of invariant measures.

\textit{\textbf{Step 1:} approximated problem (\cite[Section 4]{BZZ26}).}  The
map $g(u)=\varphi(u)u$ and $B(u)$ are truncated via a smooth cutoff
$\theta_n(\|u\|_{X_t})$ to make the fixed-point map globally Lipschitz on
$\mathbb{M}^2(X_T)$ (\cite[Propositions 4.7, 4.8]{BZZ26}). Since $G$ and $f$ do not depend
on $u$, the additive noise and the forcing terms are added without truncation, that is
\begin{equation*}\begin{aligned}
u_n(t)=S(t)u_0-\int_0^tS(t-r)\theta_n(\|u_n\|_{X_r})B(u_n(r))\,{\rm d}r
+\int_0^tS(t-r)\,f\,{\rm d}r
\\
+\int_0^tS(t-r)\theta_n(\|u_n\|_{X_r})g(u_n(r))\,{\rm d}w_0(r)
+\int_0^tS(t-r)G\,{\rm d}w(r).
\end{aligned}\end{equation*}
The extra additive-noise term is bounded using the semigroup estimates of
\cite{BZZ26}[Proposition 4.4] and exploiting Assumption [A2]:
\[
\Big\|\int_0^{\cdot}S(\cdot-r)G\,{\rm d}w(r)\Big\|^2_{\mathbb{M}^2(X_T)}
\le  C\,T\sum_{j=1}^\infty\gamma_j^2 <+\ \infty.
\]
The forcing term is even simpler to handle, since it involves no stochastic
integral at all: by \cite{BZZ26}[Proposition 4.4(i)], applied to the constant function
$f\in H$ one gets,
\[
\Big\|\int_0^{\cdot}S(\cdot-r)f\,{\rm d}r\Big\|_{X_T}^2\ \le\ C\,T\,\|f\|_H^2<+\infty.
\]
The estimates above show that the additive stochastic convolution
$\int_0^{\cdot}S(\cdot-r)G\,{\rm d}w(r)$ and the deterministic convolution $\int_0^{\cdot}S(\cdot-r)\,f{\rm d}r$ belong to $\mathbb{M}^2(X_T)$ (resp.\ $X_T$, $\Po$-a.s.), with norm
controlled by  a constant independent of $n$. Moreover these terms
 cancel exactly in the contraction estimates and hence the rest of proof remains unchanged.

\textit{\textbf{Step 2:} existence of a unique local solution.} The strategy for existence and
uniqueness of a local solution is exactly that of \cite{BZZ26}: the forcing term $f$ and the additive noise term  $G\,{\rm d}w(t)$ are state-independent, hence never subject
to the truncation $\theta_n$ and, being identical for every two solutions, cancels exactly in every pathwise-uniqueness argument.

\textit{\textbf{Step 3:}} local estimates for non explosion. The local a priori estimates
underlying non-explosion (\cite[Proposition 7.5, Lemma 7.6]{BZZ26}) must
be suitably modified to account for the additive noise and the deterministic forcing term; this is done in full details in Lemma \ref{est_in_P}.
We stress that Lemma \ref{est_in_P} is applied here to the
local maximal solution $(u,\tau)$ up to its own maximal existence
time. 

\textit{\textbf{Step 4:} existence of a unique global solution.} The conclusion of the proof, passing
from non-explosion to global existence, is again analogous to
\cite[Theorem 8.1, Theorem 8.2]{BZZ26}, since it depends only on the
functional form of the bounds established in Step 3, not on the
specific value of their constants.
\end{proof}

\bigskip 
\section*{Acknowledgments}
This work started when Z.B. was visiting the University of Pavia; he thanks the project "Dipartimento di Eccellenza" of the Department of Economics and Management for financial support. 
B.F. and M.Z. are members of Gruppo Nazionale per l’Analisi Matematica, la Probabilità e le loro Applicazioni (GNAMPA) of  Istituto Nazionale di Alta Matematica (INdAM).

\end{document}